\documentclass[preprint,12pt]{elsarticle}    

\usepackage{graphicx,epsfig}
\usepackage{enumerate}
\usepackage{amssymb}
\usepackage{amsthm}
\usepackage{amsmath}
\usepackage{centernot}
\usepackage{xcolor}
\usepackage{todonotes}
\usepackage{times}
\usepackage{mathrsfs}
\usepackage[utf8]{inputenc}
\usepackage{subfigure}
\usepackage{multirow}

\usepackage{mathtools}
\usepackage{appendix}

\usepackage{amsmath}
\usepackage{wrapfig}
\usepackage{setspace}
\usepackage{booktabs}
\usepackage{lscape}
\usepackage{algorithm,algorithmicx}
\usepackage{algcompatible}
\usepackage{bm}
\usepackage{graphicx,epsfig}
\usepackage{xfrac}
\usepackage{url}
\usepackage{multirow}
\usepackage{longtable}
\usepackage[linkcolor=blue, urlcolor=blue, citecolor=blue,
colorlinks]{hyperref}

\newtheorem{thm}{Theorem}[section]
\theoremstyle{definition}
\newtheorem{dfn}{Definition}[section]
\newtheorem{lem}{Lemma}[section]
\newtheorem{example}{Example}[section]
\newtheorem{rmrk}{Remark}[]

\makeatletter
\newcommand{\thickhline}{%
	\noalign {\ifnum 0=`}\fi \hrule height 1pt
	\futurelet \reserved@a \@xhline
}

\journal{Information Sciences}
\begin{document}
\begin{frontmatter}
\title{\textbf{Generalized Hukuhara-Clarke Derivative of Interval-valued Functions and its Properties}}
		\author[iitbhu_math]{Ram Surat Chauhan}
		\ead{rschauhan.rs.mat16@itbhu.ac.in}
		\author[iitbhu_math]{Debdas Ghosh\corref{cor1}}
		\ead{debdas.mat@iitbhu.ac.in}
		\author[address_ramik]{Jaroslav Ram{\'i}k}
		\ead{ramik@opf.slu.cz}
		\author[iitbhu_math]{Amit Kumar Debnath}
		\ead{amitkdebnath.rs.mat18@itbhu.ac.in}
		\address[iitbhu_math]{Department of Mathematical Sciences,  Indian Institute of Technology (BHU) Varanasi \\ Uttar Pradesh--221005, India}
		\address[address_ramik]{School of Business Administration in Karvina, Silesian
University Opava, Czechia}
		\cortext[cor1]{Corresponding author}
		\begin{abstract}
			 In this article, the notion of \emph{$gH$-Clarke derivative} for interval-valued functions is proposed. To define the concept of $gH$-Clarke derivatives, the concepts of limit superior, limit inferior, and sublinear interval-valued functions are studied in the sequel. The upper $gH$-Clarke derivative of a $gH$-Lipschitz  \emph{interval-valued function} (IVF) is observed to be a  sublinear IVF. It is found that every $gH$-Lipschitz continuous function is upper $gH$-Clarke differentiable. 
			For a convex and $gH$-Lipschitz IVF, it is shown that the upper $gH$-Clarke derivative coincides with the $gH$-directional derivative. The entire study is supported by suitable illustrative examples.\\
		\end{abstract}

		\begin{keyword}
			 Interval-valued functions \sep  Upper $gH$-Clarke derivative \sep Sublinear IVF\sep $gH$-Lipschitz function. 
			\\ \vspace{0.7cm}
		\end{keyword}

\end{frontmatter}

\section{Introduction}

	Clarke derivative \cite{dutta2005generalized}  is applied in the nonsmooth analysis where the functions do not have a unique linear approximation. Advances of nonsmooth analysis \cite{clarke1990optimization,schirotzek2007nosmooth} show the essential need of this derivative to handle nondifferentiable functions, especially in the absence of convexity. The topics of optimization \cite{Jahn2007}, control theory \cite{Jahn2007}, variational method \cite{Ansari2013}, etc.\ are wide application areas of Clarke derivative. \\

    As the topic of this study is Clarke derivatives for IVFs, in the following subsection, we describe the literature on Interval Optimization Problems (IOPs) and calculus of Interval-Valued Functions (IVFs). The analysis of IVFs enables one to effectively deal with the errors/uncertainties that appear while modeling practical problems. It is to be noted that there is a relatively large joint intersection of the literature survey of this paper with the recent paper by Ghosh et al. \cite{Ghosh2019derivative}. However, the proposed work in this paper is completely different than that in \cite{Ghosh2019derivative}. In  \cite{Ghosh2019derivative}, the properties of IVFs that are differentiable in the sense of $gH$-directional, $gH$-G\^{a}teaux, and   $gH$-Fr\'{e}chet derivatives have been studied. On the other hand, in this paper, we attempt to study the properties of the IVFs that are upper $gH$-Clarke differentiable.  \\

	\subsection{Literature Survey}

    In the literature of interval analysis, initially, Moore \cite{Moore1966}  developed interval arithmetic to deal with compact intervals and IVFs. In  Moore's interval arithmetic, there are a few limitations (see \cite{Ghosh2019derivative} for details), such as, the additive inverse of a nondegenerate interval, i.e., an interval whose upper and lower limits are unequal, does not exist. For the same reason, many conventional properties for real numbers are not true for compact intervals, for instance, for two compact intervals $\textbf{A}$ and $\textbf{B}$, $\textbf{A}\oplus (\textbf{B}\oplus (-\textbf{A})) \neq \textbf{B}~(\text{see Remark \ref{rmk}})$. Thus, to develop a theoretical framework of the calculus of IVFs and interval analysis, a new rule for the difference of compact intervals is introduced by Hukuhara \cite{Hukuhara1967}, known as Hukuhara difference ($H$-difference) of intervals. Although $H$-difference provides the additive inverse of compact intervals, this difference of a compact interval $\textbf{B}$ from a compact interval $\textbf{A}$ can be calculated only when the width of $\textbf{A}$ is greater than or equal to that of $\textbf{B}$ \cite{Chalco2013-2}. 
    To overcome this difficulty, a nonstandard difference of intervals is introduced is \cite{Markov1979} which is named as generalized Hukuhara difference ($gH$-difference) of intervals by Stefanini \cite{Stefanini2008}. The $gH$-difference provides an additive inverse of any compact interval and is applicable for all pairs of compact intervals. Apart from Moore's interval arithmetic, another concept of interval arithmetic has been developed by  Piegat and Landowski \cite{Landowski2015}, namely RDM interval arithmetic, which also ensures the existence of an additive inverse for any compact interval. Generally, all the properties of RDM interval arithmetic are similar to Moore's interval arithmetic except the subtraction of an interval from itself. In this article, we use Moore's interval arithmetic with $gH$-difference instead of RDM interval arithmetic (see Note 2 in  \cite{Ghosh2019derivative} for the reason).\\
	
In the study of interval analysis, in addition to interval arithmetic, an appropriate ordering of intervals and the calculus of IVFs play key roles. Unlike the real numbers, intervals are not linearly ordered. Thus, the development of optimization theory with interval-valued function is not a trivial extension of the conventional optimization theory. Most often \cite{Chalco2013-2, Ghosh2019derivative, Stefanini2009, Wu2008}, IOPs have been analyzed with respect to a partial ordering \cite{Ishibuchi1990}. Some researchers \cite{Bhurjee2016, Ghosh2017spc} used ordering relations of intervals based on the parametric comparison of intervals. In \cite{Costa2015}, an ordering relation of intervals is defined by a bijective map from the set of intervals to the Euclidean plane $\mathbb{R}^2$. However, these ordering relations \cite{Bhurjee2016, Ghosh2017spc, Costa2015} of intervals can be derived from the relations described in \cite{Ishibuchi1990}. Sengupta et al. \cite{Sengupta2001} proposed an acceptability function for intervals, just like a fuzzy membership function. Recently, Ghosh et al. \cite{ghosh2020ordering} investigated variable ordering relations for intervals and used them in IOPs. \\

To observe the properties of an IVF, calculus plays an essential role. Initially, in order to develop the calculus of IVFs, Hukuhara  \cite{Hukuhara1967} introduced the concept of differentiability of IVFs with the help of Hukuhara difference of intervals. However, the definition of Hukuhara differentiability is restrictive \cite{Chalco2013-2}. In general, if $\textbf{F}(x) = \textbf{A} \odot f(x)$, where $\textbf{A}$ is a compact interval and $f(x)$ is a real-valued function, then $\textbf{F}$ is not Hukuhara differentiable in the case of $f'(x)<0$ \cite{Bede2005}. In order to refine the calculus of IVFs, the concepts of $gH$-derivative, $gH$-partial derivative, $gH$-gradient, and $gH$-differentiability for IVFs have been developed in \cite{Chalco2013-2,Ghosh2016newton, Markov1979,  Stefanini2009, Stefanini2019}. Recently, Ghosh et al.\ \cite{Ghosh2019derivative} have provided the idea of $gH$-directional derivative, $gH$-G\^{a}teaux derivative, and $gH$-Fr\'echet derivative of IVFs.\\

\subsection{Motivation and Contribution}
	
 From the literature on the analysis of IVFs, one can notice that the study of traditional generalized derivative (Clarke derivative) for IVFs have not been developed so far. However,  the basic properties of generalized derivatives might be beneficial for characterizing and capturing the optimal solutions of IOPs with nonsmooth IVFs. To define and find properties of Clarke derivative of IVFs, we need to establish the notions of limit superior and sublinearity for IVFs. In this article, after illustrating the concept of limit superior, limit inferior, and sublinearity of IVFs, we define upper and lower $gH$-Clarke derivatives of IVFs. Although both of the upper and lower $gH$-Clarke derivatives of IVFs are defined in this article, only the properties of the upper $gH$-Clarke derivative are studied since the results for the lower derivative can be used analogously. It is shown that if an IVF is upper $gH$-Clarke differentiable at a point, then its derivative is a sublinear IVF. We further prove that the upper $gH$-Clarke derivative exists at any point if the IVF is convex $gH$-Lipschitz and the derivative is equal to $gH$-directional derivative. \\

	\subsection{Delineation}
	
    The rest of the article is demonstrated in the following sequence. The next section covers some basic terminology and notions of convex analysis and interval analysis, followed by the convexity and calculus of IVFs that are required in this paper. Also, a few properties of intervals, the $gH$-directional of an IVF is discussed in Section \ref{section1}. The concepts of limit superior, sublinear IVF and their properties are given in Section \ref{section2}. In the same section, we define upper $gH$-Clarke derivative, lower $gH$-Clarke derivative of IVFs, and prove that the upper $gH$-Clarke derivative of $gH$-Lipschitz IVF always exists.
    Also, for convex $gH$-Lipschitz continuous IVF, it is shown that upper $gH$-Clarke derivative coincides with $gH$-directional derivative. Further, the sublinearity of the upper $gH$-Clarke derivative is shown in the same section.

\section{\textbf{Preliminaries and Terminology}}\label{section1}

This section is devoted to some basic notions on intervals and the convexity and calculus of IVFs. Throughout the paper, we use the following notations.
	
	\begin{itemize}
	\item $\mathcal{X}$ denotes a  real normed linear space with the norm $\|\cdot\|$
	
	\item $\mathcal{B}(\bar{x}, \delta)$ denotes the open ball centered at $\bar{x} \in \mathcal{X}$ with radius $\delta$
	
	\item $\overline{\mathcal{B}}(\bar{x}, \delta)$ denotes $\mathcal{B}(\bar{x}, \delta) \setminus\{\bar{x}\}$
	
	
	
	
	
	\item $\mathbb{R}$ denotes the set of real numbers
    \item $\mathbb{R}_{+}$ denotes the set of nonnegative real numbers
    
	\end{itemize}
	
	\subsection{Arithmetic of Intervals and their Dominance Relation}\label{ssai}

	In this section, we discuss Moore's interval arithmetic \cite{Moore1966, Moore1987} followed by the concepts of $gH$-difference of two intervals and ordering of intervals \cite{Ishibuchi1990}. \\

	 Throughout the article, we denote the set of closed and bounded intervals by $I(\mathbb{R})$ and the elements of $I(\mathbb{R})$ by bold capital letters: ${\textbf A}, {\textbf B}, {\textbf C}, \ldots $. We represent an element $\textbf{A}$ of $I(\mathbb{R})$ in its interval form with the help of the corresponding small letter in the following way: 
	\[
	\textbf{A} = [\underline{a}, \overline{a}],~\text{where}~\underline{a}~\text{and}~\overline{a}~\text{are real numbers such that}~\underline{a} \leq \overline{a}.
	\]
	It is noteworthy that any singleton $\{p\}$ of $\mathbb{R}$ or, a real number can be represented by an interval $\textbf{P}=[\underline{p},\;\overline{p}]$, where $\underline{p}=p=\overline{p}$. In particular,
	\[
	\textbf{0}=\{0\}=[0, 0]~\text{and}~\textbf{1}=\{1\}=[1, 1].
	\]
	Consider two intervals ${\textbf A} = [\underline{a}, \overline{a}]$ and $\textbf{B} = \left[\underline{b}, \overline{b}\right]$. The \emph{addition} of $\textbf{A}$ and $\textbf{B}$, denoted by $\textbf{A} \oplus \textbf{B},$ is defined by
	\[
	\textbf{A} \oplus \textbf{B} = \left[~\underline{a} + \underline{b},~ \overline{a} +
	\overline{b}~\right].
	\]
	\noindent The \emph{subtraction} of $\textbf{B}$ from $\textbf{A}$, denoted by
	$\textbf{A} \ominus \textbf{B}$, is defined by
	\[
	\textbf{A} \ominus \textbf{B} = \left[~\underline{a} - \overline{b},~ \overline{a} -
	\underline{b}~\right].
	\]
	The \emph{multiplication} of $\textbf{A}$ and $\textbf{B}$, denoted by $\textbf{A} \odot \textbf{B}$, is defined by
	\[
	\textbf{A} \odot \textbf{B}  = \left[\min\left\{\underline{a}\; \underline{b},\; \underline{a}\overline{b},\; \overline{a}\underline{b},\; \overline{a}\overline{b}\right\},\; \max\left\{\underline{a}\; \underline{b},\; \underline{a}\overline{b},\; \overline{a}\underline{b},\; \overline{a}\overline{b}\right\} \right].
	\]
	The multiplication of $\textbf{A}$ by a real constant $\lambda$, denoted  by $\lambda \odot  \textbf{A}$ or $\textbf{A} \odot \lambda$, is defined by
	\[
	\lambda \odot  \textbf{A} =
	\begin{cases}
	[\lambda \underline{a},~\lambda \overline{a}] & \text{if $\lambda \geq 0$}\\
	[\lambda \overline{a},~\lambda \underline{a}] & \text{if $\lambda < 0.$}
	\end{cases}
	\]
	Notice that the definition of $\lambda \odot  \textbf{A}$ follows from the fact $\lambda = [\lambda, \lambda]$ and the definition of multiplication $\textbf{A} \odot \textbf{B}$.\\ \\
	Let $0 \not\in \textbf{B}$. The \emph{division} of $\textbf{A}$ by $\textbf{B}$, denoted by $\textbf{A} \oslash \textbf{B}$, is defined by
	\[
	\textbf{A} \oslash \textbf{B} = \left[\min\left\{\underline{a}/\underline{b},\; \underline{a}/\overline{b},\; \overline{a}/\underline{b},\; \overline{a}/\overline{b}\right\},\; \max\left\{\underline{a}/\underline{b},\; \underline{a}/\overline{b},\; \overline{a}/\underline{b},\; \overline{a}/\overline{b}\right\} \right].
	\]


	Since $\textbf{A}\ominus\textbf{A}\neq\textbf{0}$ for any nondegenerate interval $\textbf{A}$, we use the following concept of difference of intervals in this article.

	\begin{dfn}(\emph{$gH$-difference of intervals} \cite{Stefanini2008}). Let $\textbf{A} = [\underline{a}, \overline{a}]$ and $\textbf{B} = [\underline{b}, \overline{b}]$ be two elements of $I(\mathbb{R})$. The \emph{$gH$-difference}
		between $\textbf{A}$ and $\textbf{B}$, denoted by $\textbf{A} \ominus_{gH} \textbf{B}$, is defined by the interval $\textbf{C}$ such that
		\[
		\textbf{A} =  \textbf{B} \oplus  \textbf{C} ~\text{ or }~ \textbf{B} = \textbf{A}
		\ominus \textbf{C}.
		\]
		It is to be noted that for $\textbf{A} = \left[\underline{a},~\overline{a}\right]$ and $\textbf{B} = \left[\underline{b},~\overline{b}\right]$,
		\[
		\textbf{A} \ominus_{gH} \textbf{B} = \left[\min\{\underline{a}-\underline{b},
		\overline{a} - \overline{b}\},~ \max\{\underline{a}-\underline{b}, \overline{a} -
		\overline{b}\}\right] \text{ and } \textbf{A} \ominus_{gH} \textbf{A} = \textbf{0}.
		\]
	\end{dfn}
	 In the following, we provide a domination relation on intervals that is used throughout the paper. We remark that \emph{domination} in the following definition is  based on a \emph{minimization} type optimization problems: the \emph{smaller value} the \emph{better}.
	 \begin{dfn}\label{interval_dominance}
	 	(\emph{Dominance of intervals} \cite{Wu2007}).  Let $\textbf{A} = [\underline{a}, \overline{a}]$ and $\textbf{B} = [\underline{b}, \overline{b}]$
	 	be two intervals in $I(\mathbb{R})$.
	 	\begin{enumerate}[(i)]
	 		\item $\textbf{B}$ is said to be \emph{dominated by} $\textbf{A}$ if $\underline{a}~\leq~ \underline{b}$ and $\overline{a}~\leq~\overline{b}$, and then we write $\textbf{A}~\preceq~ \textbf{B}$;
	 		\item $\textbf{B}$ is said to be \emph{strictly dominated by} $\textbf{A}$ if either `$\underline{a} \leq \underline{b}$  and $\overline{a} < \overline{b}$' or `$\underline{a} < \underline{b}$  and $\overline{a} \leq \overline{b}$', and then we write $\textbf{A}\prec \textbf{B}$;
	 		\item if $\textbf{B}$ is not dominated by $\textbf{A}$, then we write $\textbf{A}\npreceq \textbf{B}$; if $\textbf{B}$ is not strictly dominated by $\textbf{A}$, then we write $\textbf{A}\nprec \textbf{B}$;
	 			\item if either $\textbf{A}$ is dominated by $\textbf{B}$ or $\textbf{B}$ is dominated by $\textbf{A}$, then it will be said that $\textbf{A}$ and $\textbf{B}$ are \emph{comparable},
	 		\item if $\textbf{A}\npreceq \textbf{B}$ and $\textbf{B}~\npreceq~ \textbf{A}$, then it will be said that \emph{none of $\textbf{A}$ and $\textbf{B}$ dominates the other}, or $\textbf{A}$ and $\textbf{B}$ are \emph{not comparable}.
	 	
	 	\end{enumerate}
	 \end{dfn}
	Notice that if $\textbf{B}$ is strictly dominated by $\textbf{A}$, then $\textbf{B}$ is dominated by $\textbf{A}$. Moreover, if $\textbf{B}$ is not dominated by $\textbf{A}$, then $\textbf{B}$ is not strictly dominated by $\textbf{A}$. 
	
		

	\subsection{Few Properties of Intervals}
	In this subsection, a few properties of the elements of $I(\mathbb{R})$ is studied that are used later in the paper. In the rest of the paper, by the norm of an interval, we refer to the following definition.
	\begin{dfn}\label{irnorm}
		(\emph{Norm on $I(\mathbb{R})$} \cite{Moore1966}). Let $\textbf{A}= [\underline{a}, \overline{a}]$. A function ${\lVert \cdot \rVert}_{I(\mathbb{R})} : I(\mathbb{R}) \rightarrow \mathbb{R}_{+}$, defined by
		\[
		{\lVert \textbf{A} \rVert}_{I(\mathbb{R})} = {\left\lVert \left[\underline{a}, \bar{a}\right] \right\rVert}_{I(\mathbb{R})} = \max \{|\underline{a}|, |\bar{a}|\},
		\]
		is a \emph{norm} $I(\mathbb{R})$.
	\end{dfn}
	
	
	

\begin{lem}\label{last}
	For all $x,~y \in \mathbb{R}$ and $\textbf{C} \in I(\mathbb{R})$,
	\begin{enumerate}[(i)]
	    \item \label{c1} if $\textbf{C}\succeq \textbf{0}$, then $\lvert x+y \rvert \odot \textbf{C} \preceq \lvert x \rvert \odot \textbf{C} \oplus  \lvert y \rvert \odot \textbf{C} $,
	
	    \item \label{c2}if $\textbf{C}\preceq \textbf{0}$, then $\lvert x+y \rvert \odot \textbf{C} \succeq \lvert x \rvert \odot \textbf{C} \oplus  \lvert y \rvert \odot \textbf{C},~\text{and} $
	
	\end{enumerate}
	\end{lem}
	
	\begin{proof}
	 	See  \ref{appendix_co}.
	\end{proof}

		\begin{lem}\label{forfrechet}
		For all $\textbf{A},~ \textbf{B}$,  $\textbf{C}$, $\textbf{D}\in I(\mathbb{R})$,
		\begin{enumerate}[(i)]
			\item \label{aaaassaa} $(\textbf{A}\ominus_{gH}\textbf{C})\oplus(\textbf{C}\ominus_{gH}\textbf{B}) \nprec \textbf{A}\ominus_{gH}\textbf{B}$,
		    \item \label{part21} $\textbf{B} \preceq \textbf{A}\oplus [L, L], ~\text{where}~ L=\lVert \textbf{B}\ominus_{gH}\textbf{A} \rVert_{I(\mathbb{R})},~\text{and}$
			
			\item \label{part22} ${\lVert(\textbf{A}\ominus_{gH}\textbf{B})\ominus_{gH} (\textbf{C}\ominus_{gH}\textbf{D}) \rVert}_{I(\mathbb{R})}
			\leq
			\lVert \textbf{A}\ominus_{gH}\textbf{C} \rVert_{I(\mathbb{R})} \oplus \lVert \textbf{B}\ominus_{gH}\textbf{D} \rVert_{I(\mathbb{R})}.$
		\end{enumerate}
	\end{lem}

	\begin{proof}
		See  \ref{appendix frechet}.
	\end{proof}

	\begin{rmrk}\label{rmk}
      The following two points are noticeable. 
      \begin{enumerate}[(i)]
			\item For two elements $\textbf{A}$ and \textbf{B} of $I(\mathbb{R})$, if $\textbf{B}=\textbf{A} \oplus (\textbf{B} \ominus_{gH} \textbf{A} )$, then (\ref{part21}) of Lemma \ref{forfrechet} is an obvious property since $\textbf{B} \ominus_{gH} \textbf{A}\preceq [L, L]$. However, $(\textbf{A} \oplus (\textbf{B} \ominus_{gH} \textbf{A} ))$ is not always equal to $\textbf{B}$. For instance, for $\textbf{A} = [4, 10]$ and $\textbf{B} = [-3, 2]$,
\[\textbf{A} \oplus (\textbf{B} \ominus_{gH} \textbf{A} ) = [4, 10] \oplus [-8, -7] = [-4, 3] \neq \textbf{B}.\]
Therefore, (\ref{part21}) of Lemma \ref{forfrechet} is not a trivial property.
	\item For any $\textbf{A}$, \textbf{B} and $\textbf{C}$ in $I(\mathbb{R})$, if
 \begin{equation}\label{eee}
  \textbf{B}\ominus_{gH}\textbf{A} \preceq \textbf{C} \implies \textbf{B}  \preceq  \textbf{A} \oplus \textbf{C},
 \end{equation}
 then replacing $\textbf{C}$ by $[L, L]$, we see that (\ref{part21}) of Lemma \ref{forfrechet} is an obvious property. However, \eqref{eee} is not always true.
		For instance, if  $\textbf{B}=[-3, 2],~ \textbf{A}=[4, 10]$ and $\textbf{C}= [-7.5, -6]$, then \[\textbf{B}\ominus_{gH}\textbf{A} = [-8, -7]~\text{and}~\textbf{A}\oplus \textbf{C} = [-3.5, 4].\] Hence, $\textbf{B}\ominus_{gH}\textbf{A} \preceq \textbf{C}$, but $\textbf{B}$ and $\textbf{A} \oplus \textbf{C}$ are not comparable. Thus, (\ref{part21}) of Lemma \ref{forfrechet} is not an obvious property.
      \end{enumerate}
	\end{rmrk}
	%
	%


	\subsection{Convexity and Calculus of IVFs}

A function $\textbf{F}$ from a nonempty subset $\mathcal{S}$ of $\mathcal{X}$ to $I(\mathbb{R})$ is known as an IVF. For each argument point $x \in \mathcal{S}$, $\textbf{F}$ can be presented by intervals
	\[
	\textbf{F}(x)=\left[\underline{f}(x),\
	\overline{f}(x)\right],
	\]
	where $\underline{f}$ and $\overline{f}$ are real-valued functions on $\mathcal{S}$ such that $\underline{f}(x) \leq \overline{f}(x)$ for all $x \in \mathcal{S}$.\\
	If $\mathcal{S}$ is convex, then the IVF $\textbf{F}$
	is said to be convex \cite{Wu2007} on $\mathcal{S}$ 	if for any $x_1$, $x_2\in\mathcal{S}$,
		\[
		\textbf{F}(\lambda_1 x_1+\lambda_2 x_2)\preceq
		\lambda_1\odot\textbf{F}(x_1)\oplus\lambda_2\odot\textbf{F}(x_2) \text{ for all } \lambda_1,\lambda_2\in[0, 1] \text{ with } \lambda_1+\lambda_2=1.
		\]
		The IVF 
		$\textbf{F}$ is said to be \emph{$gH$-continuous}  \cite{Ghosh2016newton} at an interior point $\bar{x}\in\mathcal{S}$ if
		\[
		\lim_{\lVert d \rVert\rightarrow 0}\left(\textbf{F}(\bar{x}+d)\ominus_{gH}\textbf{F}(\bar{x})\right)=\textbf{0}.
		\]
		 If $\textbf{F}$ is $gH$-continuous at each $x$ in $\mathcal{S}$, then $\textbf{F}$ is said to be $gH$-continuous on $\mathcal{S}$.\\
		 The IVF  $\textbf{F}$ is said to be \emph{$gH$-Lipschitz continuous} \cite{Ghosh2019derivative} at $\bar{x}\in \mathcal{S}$ if there exist constants $K'>0 $ and $\delta > 0$ such that
		\[ {\lVert \textbf{F}(\bar{x}) \ominus_{gH} \textbf{F}(y) \rVert }_{I(\mathbb{R})} \le K' {\lVert \bar{x}-y \rVert} ~\text{for all}~y \in \mathcal{S} \cap \mathcal{B}(\bar{x}, \delta). \]
		The constant $K'$ is called a Lipschitz constant of $\textbf{F}$ at $\bar{x}$. If there exists a $K>0 $ such that
		\[ {\lVert \textbf{F}(x) \ominus_{gH} \textbf{F}(y) \rVert }_{I(\mathbb{R})} \le K {\lVert x-y \rVert} ~\text{for all}~x,y \in \mathcal{S}, \]
	then the IVF  $\textbf{F}$ is said to be $gH$-Lipschitz continuous on $\mathcal{S}$ and the constant $K$ is said to be a Lipschitz constant of $\textbf{F}$ on $\mathcal{S}$.
	\begin{lem}(See \label{lc1}\cite{Wu2007}).
		$\textbf{F}$ is a convex IVF on a convex set $\mathcal{S} \subseteq \mathcal{X}$ if and only if $\underline{f}$
		and $\overline{f}$ are convex on $\mathcal{S}$.
	\end{lem}
\begin{lem}\label{lc2} Let $\textbf{F}$ be an IVF  on a nonempty subset $\mathcal{S}$ of $\mathcal{X}$.
\begin{enumerate}[(i)]
\item \label{lcc} $\textbf{F}$ is $gH$-continuous on $\mathcal{S}$ if and only if
		$\underline{f}$ and $\overline{f}$ are continuous on $\mathcal{S}$.
\item\label{lcrs}
$\textbf{F}$ is $gH$-Lipschitz continuous on $\mathcal{S}$ if and only if
$\underline{f}$ and $\overline{f}$ are Lipschitz continuous on $\mathcal{S}$.
\item\label{lcrs1}
If $\textbf{F}$ is a $gH$-Lipschitz continuous on $\mathcal{S}$, then $\textbf{F}$ is $gH$-continuous on $\mathcal{S}$.
\end{enumerate}		
\end{lem}

\begin{proof}
        See \ref{aplc1}.
\end{proof}

A consequence of Lemma \ref{lc2}  is that $gH$-continuity and $gH$-Lipschitz continuity of IVFs can be defined classically, i.e., without the concept of $gH$-difference. 
Then, the prefix gH- in continuity and Lipschitz continuity could be omitted.

\begin{rmrk}
Converse of \eqref{lcrs1} of Lemma \ref{lc2} is not true. For example, consider $\mathcal{X}$ as the Euclidean space $\mathbb{R}$, $\mathcal{S}=[0, 10]$, and the IVF $\textbf{F}:\mathcal{S}\to I(\mathbb{R})$, which is defined by
\[\textbf{F}(x)=\sqrt{x}\odot[2, 5].\]
Since $\underline{f}(x)=2\sqrt{x}$ and $\overline{f}(x)=5\sqrt{x}$ are continuous on $\mathcal{S}$, $\textbf{F}$ is $gH$-continuous on $\mathcal{S}$ by \eqref{lcc} of Lemma \ref{lc2}. If $\textbf{F}$ is $gH$-Lipschitz continuous on $\mathcal{S}$, then by \eqref{lcrs} of Lemma \ref{lc2}, $\underline{f}$ and $\overline{f}$ are Lipschitz continuous on $\mathcal{S}$, which is not true. Consequently, $\textbf{F}$ is not $gH$-Lipschitz continuous on $\mathcal{S}$.
\end{rmrk}

	\begin{dfn}\label{ddd}
		(\emph{$gH$-directional derivative} \cite{Ghosh2019derivative, Stefanini2019}).
		Let $\textbf{F}$ be an IVF on a nonempty subset $\mathcal{S}$ of $\mathcal{X}$. Let $\bar{x} \in \mathcal{S}$ and $h \in \mathcal{X}$. If the limit
		\[
		\lim_{\lambda \to 0+}\frac{1}{\lambda}\odot\big(\textbf{F}(\bar{x}+\lambda h)\ominus_{gH}\textbf{F}(\bar{x})\big)
		\]
		exists finitely, then the limit is said to be \emph{$gH$-directional derivative} of $\textbf{F}$ at $\bar{x}$ in the direction $h$, and it is denoted by $\textbf{F}_\mathscr{D}(\bar{x})(h)$.
	\end{dfn}

	\section{\textbf{$gH$-Clarke Derivative of IVF}} \label{section2}

	In this section, extended concepts of the $gH$-directional derivative, namely upper and lower $gH$-Clarke derivatives, for IVFs are given. A short discussion of the required notions of limit superior and sublinearity for IVFs is provided.

	\begin{dfn}\label{sup}(\emph{Supremum and limit superior of an IVF}). Let $\mathcal{S}$ be a nonempty subset of $\mathcal{X}$ and $\textbf{F}: \mathcal{S} \rightarrow I(\mathbb{R})$ be an IVF. Then, the \emph{supremum} of \textbf{F} over $\mathcal{S}$ is defined by
	\[\sup_{\mathcal{S}} \textbf{F} = \left[\sup_{\mathcal{S}} \underline{f},~ \sup_{\mathcal{S}}\overline{f}\right],\]
	where $\sup\limits_{\mathcal{S}}~ \underline{f} = \sup\left\{\underline{f}(x) : x\in \mathcal{S}\right\}$ and $\sup\limits_{ \mathcal{S}} \overline{f} = \sup\left\{\overline{f}(x) : x\in \mathcal{S}\right\}$.\\
    The \emph{limit superior} of $\textbf{F}$ at a limit point $\bar{x}$ in $\mathcal{S}$ is defined by
	\[\limsup_{x\to \bar{x}} \textbf{F}(x) = \left[\limsup_{x\to \bar{x}} \underline{f}(x),~ \limsup_{x\to \bar{x}} \overline{f}(x)\right],\]
	where $\limsup\limits_{x\to \bar{x}} \underline{f}(x)= \lim\limits_{\delta \to 0}\bigg(\sup\limits_{~x\in \overline{\mathcal{B}}(\bar{x}, \delta) \cap \mathcal{S}}\underline{f}(x)\bigg)$~and~  $\limsup\limits_{x\to \bar{x}} \overline{f}(x)= \lim\limits_{\delta \to 0}\bigg(\sup\limits_{~x\in \overline{\mathcal{B}}(\bar{x}, \delta) \cap \mathcal{S}}\overline{f}(x)\bigg)$.
	\end{dfn}
	\begin{dfn}\label{inf}(\emph{Infimum and limit inferior of an IVF}). Let $\mathcal{S}$ be a nonempty subset of $\mathcal{X}$ and $\textbf{F}: \mathcal{S} \rightarrow I(\mathbb{R})$ be an IVF. Then, the \emph{infimum} of \textbf{F} is defined by 
\[\inf_{\mathcal{S}} \textbf{F} = \left[\inf_{x \in \mathcal{S}} \underline{f},~ \inf_{x \in \mathcal{S}}\overline{f}\right],\]
where $\inf\limits_{\mathcal{S}} \underline{f} = \inf\left\{\underline{f}(x) : x\in \mathcal{S}\right\}$ and $\inf\limits_{\mathcal{S}} \overline{f} = \inf\left\{\overline{f}(x) : x\in \mathcal{S}\right\}$.\\
    The \emph{limit inferior} of $\textbf{F}$ at a limit point $\bar{x}$ in $\mathcal{S}$ is defined by
	\[\liminf_{x\to \bar{x}} \textbf{F}(x) = \left[\liminf_{x\to \bar{x}} \underline{f}(x),~ \liminf_{x\to \bar{x}} \overline{f}(x)\right],\]
	where $\liminf\limits_{x\to \bar{x}} \underline{f}(x)= \lim\limits_{\delta \to 0}\bigg(\inf\limits_{~x\in \overline{\mathcal{B}}(\bar{x}, \delta) \cap \mathcal{S}}\underline{f}(x)\bigg)$ and  $\liminf\limits_{x\to \bar{x}} \overline{f}(x)= \lim\limits_{\delta \to 0}\bigg(\inf\limits_{~x\in \overline{\mathcal{B}}(\bar{x}, \delta) \cap \mathcal{S}}\overline{f}(x)\bigg)$.
		
	\end{dfn}

	\begin{lem}\label{00}
	Let $\mathcal{S}$ be a nonempty subset of $\mathcal{X}$ and $\textbf{F},~\textbf{G}: \mathcal{S} \rightarrow I(\mathbb{R})$ be two IVFs. Then, at any $\bar{x}\in \mathcal{S}$, the following properties are true:
	\begin{enumerate}[(i)]
	    \item \label{01} $\limsup\limits_{x\to \bar{x}}\left(\textbf{F}(x) \oplus \textbf{G}(x)\right) \preceq \limsup\limits_{x\to \bar{x}} \textbf{F}(x) \oplus \limsup\limits_{x\to \bar{x}} \textbf{G}(x)$,
			
		\item \label{02} $\limsup\limits_{x\to \bar{x}}\left(\lambda \odot \textbf{F}(x)\right) =  \lambda \odot \limsup\limits_{x\to \bar{x}} \textbf{F}(x)~\text{for all}~\lambda \geq 0$, and
		
		\item \label{03} $\left\lVert \limsup\limits_{x\to \bar{x}} \textbf{F}(x) \right\rVert_{I(\mathbb{R})} \leq \limsup\limits_{x\to \bar{x}} \left\lVert  \textbf{F}(x) \right\rVert_{I(\mathbb{R})}. $
			
	\end{enumerate}
	\end{lem}
	
	\begin{proof}
		See \ref{appendix_sup}.
	\end{proof}

	\begin{dfn}\label{dffff}
		(\emph{Upper $gH$-Clarke derivative}).
		Let $\textbf{F}$ be an IVF on a nonempty subset $\mathcal{S}$ of $\mathcal{X}$. For   $\bar{x}\in \mathcal{S}$ and $h \in \mathcal{X}$, if the limit superior
		\[
		\limsup_{\substack{%
				x \to \bar{x}\\
				\lambda \to 0+}}\frac{1}{\lambda}\odot\big(\textbf{F}(x+\lambda h)\ominus_{gH}\textbf{F}(x)\big)=\lim\limits_{\delta \to 0}\bigg(\sup\limits_{x\in \overline{\mathcal{B}}(\bar{x}, \delta) \cap \mathcal{S},~\lambda \in (0, \delta)}\frac{1}{\lambda}\odot\big(\textbf{F}(x+\lambda h)\ominus_{gH}\textbf{F}(x)\big) \bigg)
		\]
		 exists finitely, then the limit superior value is called \emph{upper $gH$-Clarke derivative} of $\textbf{F}$ at $\bar{x}$ in the direction $h$, and it is denoted by $\textbf{F}_\mathscr{C}(\bar{x})(h)$. If this limit superior exists for all $h \in \mathcal{X}$, then $\textbf{F}$ is said to be \emph{upper $gH$-Clarke differentiable} at $\bar{x}$.
	\end{dfn}

		\begin{dfn}\label{dfffl}
			(\emph{Lower $gH$-Clarke derivative}).
			Let $\textbf{F}$ be an IVF on a nonempty subset $\mathcal{S}$ of $\mathcal{X}$. For   $\bar{x}\in \mathcal{S}$ and $h \in \mathcal{X}$, if the limit inferior
			\[
			\liminf_{\substack{%
					x \to \bar{x}\\
					\lambda \to 0+}}\frac{1}{\lambda}\odot\big(\textbf{F}(x+\lambda h)\ominus_{gH}\textbf{F}(x)\big) =\lim\limits_{\delta \to 0}\bigg(\inf\limits_{~x\in \overline{\mathcal{B}}(\bar{x}, \delta) \cap \mathcal{S},~\lambda \in (0, \delta)}\frac{1}{\lambda}\odot\big(\textbf{F}(x+\lambda h)\ominus_{gH}\textbf{F}(x)\big) \bigg)
			\]
			exists finitely, then the limit inferior value is called \emph{lower $gH$-Clarke derivative} of $\textbf{F}$ at $\bar{x}$ in the direction $h$. If this limit inferior exists for all $h \in \mathcal{X}$, then $\textbf{F}$ is said to be \emph{lower $gH$-Clarke differentiable} at $\bar{x}$.
		\end{dfn}

If $\textbf{F}$ has both upper and lower $gH$-Clark derivatives at $x$ and they are equal, then $\textbf{F}$ is called $gH$-Clark differentiable at $x$.

\begin{rmrk}
			Conventionally, for real valued-functions, the terminologies Clarke derivative \cite{clarke1990optimization,Jahn2007} and upper Clarke derivative \cite{demyanov2002rise} are interchangeably used. In fact, the upper Clarke derivative is usually referred to as Clarke derivative. However, in order to avoid any confusion, we prefix upper and lower with the Clarke derivative corresponding to the values given by limit superior and limit inferior, respectively. In addition, throughout the article, we use the notation $\textbf{F}_\mathscr{C}$ to refer the upper $gH$-Clarke derivative of an IVF $\textbf{F}$.
		\end{rmrk}

		\begin{rmrk}
		 It is clear that \textbf{F} is lower $gH$-Clark differentiable at $\bar{x}$ if and only if  $(\ominus\textbf{F})$ is upper $gH$-Clark differentiable at $\bar{x}$. That is why we deal only with the upper $gH$-Clark differentiability in this study.
		\end{rmrk}

		\begin{example}
		 In this example, we calculate the upper $gH$-Clarke derivative at $\bar{x}=0$ for the IVF $\textbf{F}(x) = \lvert x \rvert \odot \textbf{C}$, where $\textbf{0} \preceq \textbf{C} \in I(\mathbb{R})$, $\mathcal{X}$ is the Euclidean space $\mathbb{R}$, and $\mathcal{S}=\mathcal{X}$. \\
		For any $h \in \mathcal{X}$, we see that
		\begin{eqnarray}\label{0_1}
			&& \limsup_{\substack{%
					x \to 0\\
					\lambda \to 0+}}\frac{1}{\lambda}\odot\left(\textbf{F}(x+\lambda h)\ominus_{gH}\textbf{F}(x)\right) \nonumber \\
			&\preceq& \limsup_{\substack{%
					x \to 0\\
					\lambda \to 0+}}\frac{1}{\lambda}\odot\left(\lvert x\rvert \odot \textbf{C} \oplus  \lambda \lvert h\rvert\odot \textbf{C}\ominus_{gH} \lvert x\rvert\odot \textbf{C}\right)~\text{by Lemma \ref{last}} \nonumber \\
			&=&\lvert h \rvert \odot \textbf{C}.
		\end{eqnarray}
		Further,
		\begingroup
		\allowdisplaybreaks
		\begin{eqnarray}\label{0_2}
			&& \limsup_{\substack{%
					x \to 0\\
					\lambda \to 0+}}\frac{1}{\lambda}\odot\left(\lvert x+\lambda h\rvert\odot \textbf{C}\ominus_{gH} \lvert x\rvert\odot \textbf{C}\right)
					\nonumber \\
			&\succeq&  \limsup_{\substack{%
					\lambda \to 0+}}\frac{1}{\lambda}\odot\left(2\lambda \lvert h\rvert\odot \textbf{C}\ominus_{gH} \lambda \lvert h\rvert\odot \textbf{C}\right)~(\text{taking $x=\lambda h, h>0$}) \nonumber \\
			&=&\lvert h \rvert \odot \textbf{C}.
		\end{eqnarray}
		\endgroup
		Then, from the inequalities (\ref{0_1}) and (\ref{0_2}), we get 
		$\textbf{F}_\mathscr{C}(\bar{x})(h) = \lvert h \rvert \odot \textbf{C}$.
		\end{example}

		\begin{lem}\label{clarkeimply}
		If $\underline{f}$ and $\overline{f}$ are upper Clarke differentiable at $\bar{x}\in \mathcal{S} \subseteq \mathcal{X}$, then the IVF $\textbf{F}$ is upper $gH$-Clarke differentiable at $\bar{x} \in \mathcal{S}$.
	\end{lem}
	\begin{proof}
	        See \ref{appendix_imply}.
    	\end{proof}

	\begin{rmrk}\label{00011}
		Let $\mathcal{S}$ be a nonempty subset of $ \mathcal{X}$ and the IVF $\textbf{F}: \mathcal{S} \rightarrow I(\mathbb{R})$ be upper $gH$-Clarke diffrentiable at $\bar{x} \in \mathcal{S}$. Then, $\textbf{F}$ may not be $gH$-directional differentiable at $\bar{x} \in \mathcal{S}$. For example, take $\mathcal{X}$ as the Euclidean space $\mathbb{R}$, $\mathcal{S}=\mathcal{X}$ and the IVF $\textbf{F}: \mathcal{S} \rightarrow I(\mathbb{R})$, which is defined by
		\[\textbf{F}(x) =
		\begin{cases}
	 \frac{\sin^2x}{x}\odot \textbf{C} & \text{if}~ x\neq 0\\
	5\odot \textbf{C} & \text{if}~ x=0, 
		\end{cases}
		\]
where $\textbf{C} \in I(\mathbb{R})$  with $\textbf{C}\succeq\textbf{0}$. For all nonzero $h$ in $\mathcal{X}$, we obtain
		\begingroup
	\allowdisplaybreaks
		\begin{eqnarray*}
		&&\lim _{\lambda \to 0+}\frac{1}{\lambda}\odot\left(\textbf{F}(x+\lambda h)\ominus_{gH}\textbf{F}(x)\right)\\
		&=& \lim _{\lambda \to 0+}\frac{1}{\lambda}\odot\left( \left(\frac{\sin^2(x+\lambda h)}{x+\lambda h}\right)\odot \textbf{C}\ominus_{gH}\left(\frac{\sin^2x}{x}\right)\odot \textbf{C}\right)\\
		&=& \lim _{\lambda \to 0+}\frac{1}{\lambda}\odot\left( \left(\frac{\sin^2(x+\lambda h)}{x+\lambda h}-\frac{\sin^2x}{x}\right)\odot \textbf{C}\right)\\
		&=& \left(\frac{h\sin{2x}}{x}-\frac{h\sin^2x}{x^2}\right)\odot \textbf{C}.
		\end{eqnarray*}
	Thus,
			\[\lim_{\substack{%
				x \to 0\\
				\lambda \to 0+}}\frac{1}{\lambda}\odot\left(\textbf{F}(x+\lambda h )\ominus_{gH}\textbf{F}(x)\right) = h\odot \textbf{C},\]
				which implies
					\[\limsup_{\substack{%
				x \to 0\\
				\lambda \to 0+}}\frac{1}{\lambda}\odot\left(\textbf{F}(x+\lambda h )\ominus_{gH}\textbf{F}(x)\right) = h\odot \textbf{C}.\]
				Hence, $\textbf{F}$ is upper $gH$-Clarke differentiable at $\bar{x}=0$.
				However, the limit 
				\begin{eqnarray*}\label{key}
					&&\lim _{\lambda \to 0+}\frac{1}{\lambda}\odot\left(\textbf{F}(\bar{x}+\lambda h)\ominus_{gH}\textbf{F}(\bar{x})\right) \nonumber\\
					&=& \lim _{\lambda \to 0+}\frac{1}{\lambda}\odot\left( \left(\frac{\sin^2(\lambda h)}{\lambda h}\right)\odot \textbf{C}\ominus_{gH} 5\odot \textbf{C}\right) 
				\end{eqnarray*} 
				does not exist at $\bar{x}=0$. Consequently, $\textbf{F}$ is not $gH$-directional differentiable at $\bar{x}$.
				\endgroup
		\end{rmrk}

	\begin{rmrk}\label{00012}
Let $\mathcal{S}$ be a nonempty subset of $\mathcal{X}$ and $\textbf{F}: \mathcal{S} \rightarrow I(\mathbb{R})$ has $gH$-directional derivative at $\bar x \in \mathcal{S}$. Then, $\textbf{F}$ is not necessarily upper $gH$-Clarke differentiable at $\bar x \in \mathcal{X}$. For instance, take $\mathcal{X}$ as the Euclidean space $\mathbb{R}^2$, $\mathcal{S}= \{(x_1, x_2)\in \mathbb{R}^2 : x_2\geq 0, x_2\geq 0\}$ and the IVF $\textbf{F}: \mathcal{S}  \rightarrow I(\mathbb{R})$, which is defined by
\[
	\textbf{F}(x_1,x_2)=
	\begin{cases}
	x_1^2\left(1+\frac{1}{x_2}\right) \odot [3,~8] & \text{if } x=(x_1,x_2) \neq (0, 0)\\
	\textbf{0} & \text{otherwise}.
	\end{cases}
	\]
	Then, at $\bar{x}=(0, 0)$ and $h=(h_1, h_2)\in \mathcal{X}$ such that for sufficiently small $\lambda > 0$ so that   $\bar{x}+\lambda h \in \mathcal{S}$,  we have
	\begin{eqnarray*}
	\lim _{\lambda \to 0+}\frac{1}{\lambda}\odot\left(\textbf{F}(\bar{x}+\lambda h)\ominus_{gH}\textbf{F}(\bar{x})\right) 
	&=&\begin{cases}
	\frac{h_1^2}{h_2} \odot [3,~8] & \text{if } h_2 \neq 0\\
	\textbf{0} & \text{otherwise}.
	\end{cases}
	\end{eqnarray*}
	Hence, $\textbf{F}$ has a $gH$-directional derivative at $\bar{x}$ in every direction $h \in \mathcal{X}$.\\
	Again, for $x =(x_1, x_2)\in \mathcal{S}$ and $h=(h_1, h_2)\in \mathcal{X}$, we have
		\begingroup
		\allowdisplaybreaks
		\begin{eqnarray*}
		&&\lim _{\lambda \to 0+}\frac{1}{\lambda}\odot\left(\textbf{F}(x+\lambda h)\ominus_{gH}\textbf{F}(x)\right)\\
		&=& \lim _{\lambda \to 0+}\frac{1}{\lambda}\odot\left( (x_1+\lambda h_1)^2 \left( 1+\frac{1}{x_2+\lambda h_2}\right)\odot [3,8]\ominus_{gH}	 x_1^2\left(1+\frac{1}{x_2}\right) \odot [3,~8]\right) \\
		&=&\bigg[\min \bigg\{3\left(2x_1h_1 + \frac{2x_1h_1}{x_2}-\frac{x_1^2h_2}{x_2^2}\right),~8\left(2x_1h_1 + \frac{2x_1h_1}{x_2}-\frac{x_1^2h_2}{x_2^2}\right)\bigg\},\\
		&&~\max \bigg\{3\left(2x_1h_1 + \frac{2x_1h_1}{x_2}-\frac{x_1^2h_2}{x_2^2}\right),~8\left(2x_1h_1 + \frac{2x_1h_1}{x_2}-\frac{x_1^2h_2}{x_2^2}\right)\bigg\}\bigg].
		\end{eqnarray*}
		\endgroup
 Along $x_2=mx_1$, where $m$ is any real number, \begin{align*}
			& \lim_{\substack{%
			x \to 0\\
				\lambda \to 0+}}\frac{1}{\lambda}\odot\left(\textbf{F}(x+\lambda h )\ominus_{gH}\textbf{F}(x)\right) \\
				=& \bigg[\min \left\{ 3\left(\frac{2h_1}{m}-\frac{h_2}{m^2}\right),~8\left(\frac{2h_1}{m}-\frac{h_2}{m^2}\right)\right\},
				\max \left\{ 3\left(\frac{2h_1}{m}-\frac{h_2}{m^2}\right),~8\left(\frac{2h_1}{m}-\frac{h_2}{m^2}\right)\right\}\bigg].
\end{align*}
Hence, for $h_2>0$, $\left(\frac{2h_1}{m}-\frac{h_2}{m^2}\right)\to -\infty$ as $m \to 0.$ Consequently,
				
					\[\limsup_{\substack{%
				x \to 0\\
				\lambda \to 0+}}\frac{1}{\lambda}\odot\left(\textbf{F}(x+\lambda h )\ominus_{gH}\textbf{F}(x)\right)~\text{ does not exist}.\]
This implies that $\textbf{F}$ has no upper $gH$-Clarke derivative at $\bar{x} \in \mathcal{S}$.				
		
		\end{rmrk}

	%
	%
{The following theorem extends the well-known result from \cite{Jahn2007} for Lipschitz continuous functions to $gH$-Lipschitz continuous IVFs with the help of Lemma \ref{clarkeimply}. 
} 
	
	\begin{thm}\label{thm_lipshitz_d}
		Let $\mathcal{S}$ be a nonempty subset of $\mathcal{X}$ with $\bar{x} \in \text{int}(\mathcal{S})$ and $\textbf{F}: \mathcal{S} \rightarrow I(\mathbb{R})$ be a $gH$-Lipschitz continuous IVF at $\bar{x}$ with a Lipschitz constant $K'$. Then, $\textbf{F}$ is upper $gH$-Clarke differentiable at $\bar{x}$ and
		\[\lVert \textbf{F}_\mathscr{C}(\bar{x})(h)\rVert_{I(\mathbb{R})} \leq K'\lVert h \rVert~\text{for all}~h \in \mathcal{X}. \]
	\end{thm}
	
	\begin{proof}
		Since $\textbf{F}$ is $gH$-Lipschitz continuous on $\mathcal{S}$, for any $h\in \mathcal{X}$, we get for $\lambda >0$ that
		\begin{eqnarray}\label{kk}
			\bigg\lVert \frac{1}{\lambda}\odot\left(\textbf{F}(x+\lambda h)\ominus_{gH}\textbf{F}(x)\right) \bigg\rVert _{I(\mathbb{R})} &\leq& \frac{1}{\lambda} K' \lVert x+\lambda h - x \rVert = K' \lVert  h \rVert,
		\end{eqnarray}
		if $x$ and $\lambda$ are sufficiently close to $\bar{x}$ and $0$, respectively.
		From inequality (\ref{kk}) we have 
		
			\[ \bigg\lvert \frac{1}{\lambda} \left (\underline{f}(x+\lambda h) - \underline{f}(x) \right) \bigg\rvert \leq K' \lVert  h \rVert~\text{and}~\bigg\lvert \frac{1}{\lambda} \left (\overline{f}(x+\lambda h) - \overline{f}(x) \right) \bigg\rvert \leq K' \lVert  h \rVert.\]
			Hence, the limit superior $\underline{f}_\mathscr{C}(\bar{x})(h)$ and $\overline{f}_\mathscr{C}(\bar{x})(h)$ exist at $\bar{x}$ (cf. p. 69 of \cite{Jahn2007}). By Lemma \ref{clarkeimply}, the limit superior $
		\textbf{F}_\mathscr{C}(\bar{x})(h)$ exists.\\ Furthermore, by $gH$-Lipschitz continuity of $\textbf{F}$ on $\mathcal{S}$, we have the following for all $h \in \mathcal{X}$: 
		\begingroup
		\allowdisplaybreaks
		\begin{eqnarray*}
			\lVert \textbf{F}_\mathscr{C}(\bar{x})(h)\rVert_{I(\mathbb{R})} &=& \Bigg\lVert \limsup_{\substack{%
					x \to \bar{x}\\
					\lambda \to 0+}} \frac{1}{\lambda}\odot\left(\textbf{F}(x+\lambda h)\ominus_{gH}\textbf{F}(x)\right) \Bigg\rVert _{I(\mathbb{R})}\\
			&\leq& \limsup_{\substack{%
				x \to \bar{x}\\
 					\lambda \to 0+}}\bigg\lVert \frac{1}{\lambda}\odot\left(\textbf{F}(x+\lambda h)\ominus_{gH}\textbf{F}(x)\right) \bigg\rVert_{I(\mathbb{R})}~\text{by Lemma \ref{00}}\\
 			&\leq&  K' \lVert  h \rVert \text{ by (\ref{kk})}. 
 		\end{eqnarray*}
 	\endgroup
	\end{proof}

For convex and $gH$-Lipschitz continuous IVFs, upper $gH$-Clarke derivative and $gH$-directional derivative coincide as the next theorem states. 

\begin{thm}\label{_lipshitz}
		Let $\mathcal{X}$ be convex, and the IVF $\textbf{F}: \mathcal{X} \rightarrow I(\mathbb{R})$ be convex on $\mathcal{X}$ and $gH$-Lipschitz continuous at some $\bar{x} \in \mathcal{X}$. Then, the upper $gH$-Clarke derivative of $\textbf{F}$ at $\bar{x}$ coincides with the $gH$-directional derivative of $\textbf{F}$ at $\bar{x}$ in the direction $h \in \mathcal{X}$.
		
	\end{thm}
	
	\begin{proof}
	    Since $\textbf{F}$ is a convex IVF on $\mathcal{X}$, we get by Theorem 3.1 of \cite{Ghosh2019derivative} that the $gH$-directional derivative of $\textbf{F}$ exists at $\bar{x} \in \mathcal{X}$ in every direction $h$.
	    Also, as $\textbf{F}$ is $gH$-Lipschitz continuous at $\bar{x}$, from Theorem \ref{thm_lipshitz_d}, we get that the upper $gH$-Clarke derivative of $\textbf{F}$ exists at any $\bar{x} \in \mathcal{X}$ in every direction $h$.
        Thus, by Definitions \ref{ddd} and \ref{dffff}, we observe that
		\begin{equation}\label{hy}
			\textbf{F}_\mathscr{D}(\bar{x})(h) \preceq \textbf{F}_\mathscr{C}(\bar{x})(h)~\text{for all}~h.
		\end{equation}
		For the proof of the reverse inequality, we write
		\begingroup
		\allowdisplaybreaks
		\begin{eqnarray*}
			\textbf{F}_\mathscr{C}(\bar{x})(h) &=& \limsup_{\substack{%
					x \to \bar{x}\\
					\lambda \to 0+}}\frac{1}{\lambda}\odot\left(\textbf{F}(x+\lambda h)\ominus_{gH}\textbf{F}(x)\right)\\
			&=&  \lim_{\substack{%
					\delta \to 0+\\
					\epsilon \to 0+}}\sup_{\lVert x- \bar{x}\rVert <\delta} \sup_{0<\lambda <\epsilon}\frac{1}{\lambda}\odot\left(\textbf{F}(x+\lambda h)\ominus_{gH}\textbf{F}(x)\right).
		\end{eqnarray*}
		\endgroup
		Since $\textbf{F}$ is convex on $\mathcal{X}$, Lemma 3.1 of \cite{Ghosh2019derivative} leads to the equality
		
		\[	\textbf{F}_\mathscr{C}(\bar{x})(h) =  \lim_{\substack{%
				\delta \to 0+\\
				\epsilon \to 0+}}\sup_{\lVert x- \bar{x}\rVert < \delta} \frac{1}{\epsilon}\odot\left(\textbf{F}(x+\epsilon h)\ominus_{gH}\textbf{F}(x)\right),\]
		and for an arbitrary $\alpha >0$, 
		\[	\textbf{F}_\mathscr{C}(\bar{x})(h) =  \lim_{\substack{%
				\epsilon \to 0+}}\sup_{\lVert x- \bar{x}\rVert < \epsilon \alpha} \frac{1}{\epsilon}\odot\left(\textbf{F}(x+\epsilon h)\ominus_{gH}\textbf{F}(x)\right).\]
				Because of the $gH$-Lipschitz continuity of $\textbf{F}$ at $\bar{x}$, we have for sufficiently small $\epsilon >0$ and $\|x - \bar{x}\| < \epsilon \alpha$ that
		 	\begingroup
		\allowdisplaybreaks
		\begin{eqnarray*}
			&& \left\lVert \tfrac{1}{\epsilon}\odot\left(\textbf{F}(x+\epsilon h)\ominus_{gH}\textbf{F}(x)\right) \ominus_{gH} \tfrac{1}{\epsilon}\odot\left(\textbf{F}(\bar{x}+\epsilon h)\ominus_{gH}\textbf{F}(\bar{x})\right) \right \rVert _{I(\mathbb{R})}\\
			&\leq& \left\lVert \tfrac{1}{\epsilon}\odot\left(\textbf{F}(x+\epsilon h)\ominus_{gH}\textbf{F}(\bar{x}+\epsilon h)\right)\right \rVert _{I(\mathbb{R})}
		    +\left\lVert \tfrac{1}{\epsilon}\odot\left(\textbf{F}(x)\ominus_{gH}\textbf{F}(\bar{x})\right)\right \rVert _{I(\mathbb{R})}\\
		    && \text{by (\ref{part22}) of Lemma \ref{forfrechet}}\\
			&\leq&\tfrac{1}{\epsilon} K' \lVert x-\bar{x} \rVert + \tfrac{1}{\epsilon} K' \lVert x-\bar{x} \rVert,~\text{where}~K'~\text{is the Lipschitz constant of $\textbf{F}$ at $\bar{x} \in \mathcal{X}$}\\
			&\leq& 2K'\alpha.
		\end{eqnarray*}
		 Then, by (\ref{part21}) of Lemma \ref{forfrechet}, we have
		\begin{eqnarray*}
			\textbf{F}_\mathscr{C}(\bar{x})(h) &\preceq& \lim _{\epsilon \to 0+}\frac{1}{\epsilon}\odot\left(\textbf{F}(x+\epsilon h)\ominus_{gH}\textbf{F}(x)\right)\oplus [2K\alpha, 2K\alpha]\\
			&=& \textbf{F}_\mathscr{D}(\bar{x})(h)\oplus [2K'\alpha, 2K'\alpha].
		\end{eqnarray*}
		Since $\alpha >0$ is chosen arbitrarily, we obtain
		\begin{equation}\label{yh}
			\textbf{F}_\mathscr{C}(\bar{x})(h) \preceq \textbf{F}_\mathscr{D}(\bar{x})(h)~\text{for all}~h.
		\end{equation}
		From (\ref{hy}) and (\ref{yh}), we get
		\[\textbf{F}_\mathscr{C}(\bar{x})(h) = \textbf{F}_\mathscr{D}(\bar{x})(h).\]
		\endgroup
		
	\end{proof}

	\begin{dfn}(\emph{Sublinear IVF}). \label{sublinear_ivf}
		Let $\mathcal{S}$ be a linear subspace of $\mathcal{X}$. An IVF $\textbf{F}: \mathcal{S} \rightarrow I(\mathbb{R})$ is said to be \emph{sublinear} on $\mathcal{S}$ if
		\begin{enumerate}[(i)]
			\item $\textbf{F}(\lambda x)=\lambda\odot\textbf{F}(x)~ \text{for all}~x\in \mathcal{S}~\text{and for all}~\lambda \geq 0$, and
			\item $\textbf{F}(x+y) \nsucc \textbf{F}(x)\oplus\textbf{F}(y) $ for all $x,~y\in \mathcal{S}$.
		\end{enumerate}
	\end{dfn}
	
	\begin{example}\label{ex40}
		 Let $\mathcal{X}$ be the Euclidean space $\mathbb{R}$ and $\mathcal{S}=\mathcal{X}$. Then, the IVF $\textbf{F}: \mathcal{S} \rightarrow I(\mathbb{R})$ that is defined by
		\[\textbf{F}(x) = \lvert x \rvert \odot \textbf{C},~ \text{where}~ \textbf{C} \in I(\mathbb{R})~\text{such that}~ \textbf{C} \nprec \textbf{0},\]
		is sublinear on $\mathcal{S}$. The reason is as follows.\\
	For all $x, y \in \mathcal{S}$ and $\lambda \geq 0$, we have
		\begin{enumerate}[(i)]
			\item $ \textbf{F}(\lambda x) = \lvert  \lambda x \rvert \odot \textbf{C} = \lambda \lvert   x \rvert \odot \textbf{C}= \lambda \odot \textbf{F}(x)$.
			
			\item
			\begin{eqnarray*}
				&& \lvert  x + y \rvert \odot \textbf{C} \nsucc  \big(\lvert  x  \rvert + \lvert y \rvert \big)\odot \textbf{C} ~\text{by (\ref{c1}) Lemma \ref{last}}  \\
				&\implies& \lvert  x + y \rvert \odot \textbf{C} \nsucc  \lvert  x \rvert \odot \textbf{C} \oplus \lvert  y \rvert \odot \textbf{C}~\text{since}~\lvert   x \rvert~\text{and}~ \lvert   y \rvert~\text{are nonnegative} \\
				&\implies& \textbf{F}(x+y) \nsucc \textbf{F}(x)\oplus\textbf{F}(y).
			\end{eqnarray*}
		\end{enumerate}
		Hence, $\textbf{F}$ is a sublinear IVF on $\mathcal{S}$.
		
	\end{example}

	\begin{example}\label{lemma25}
		Let $Q$ be a real positive definite matrix of order $n\times n$ and $\mathcal{S}$ be a linear subspace of $\mathcal{X}$. Consider the IVF $\textbf{F}: \mathcal{S} \rightarrow I(\mathbb{R})$, which is defined by
		\[\textbf{F}(x) = \big( \sqrt{x^TQx} \big)\odot \textbf{C},~\text{where}~\textbf{C} \nprec \textbf{0}.\]
		Then, $\textbf{F}$ is a sublinear IVF on $\mathcal{S}$. The reason is as follows.\\
		
		The function $\textbf{F}(x)$ can be written as $g(x)\odot \textbf{C}$, where $g(x)= \sqrt{ x^TQx}$. By Example 1.2.3 of \cite{Hiriart2012}, $g$ satisfies the following conditions:
		\begin{enumerate}[(a)]
			\item for $\lambda \geq 0$ and $x \in \mathcal{S}$,
			\begin{equation}\label{xqx}
				g(\lambda x) = \lambda g(x),
			\end{equation}
			and
			\item for all $x, y \in \mathcal{S}$
			\begin{equation}\label{xtqx}
			g(x+y)\leq g(x) + g(y).
			\end{equation}
		\end{enumerate}
		From (\ref{xqx}), we have
		\[g(\lambda x) \odot \textbf{C} = \lambda g(x)\odot \textbf{C},~\text{or},~ \textbf{F}(\lambda x) = \lambda \odot \textbf{F}(x).\]
		Since $\textbf{C} \nprec \textbf{0}$, from (\ref{xtqx}) and Lemma \ref{last}, we obtain
		\begin{eqnarray*}
			&&g(x+y) \odot \textbf{C} \nsucc  ( g(x) + g(y)) \odot \textbf{C} \\
			&\implies& (g(x+y)) \odot \textbf{C} \nsucc   g(x)\odot \textbf{C} \oplus g(y) \odot \textbf{C}~\text{since}~g(x)~\text{and}~ g(y)~\text{are nonnegative} \\
			&\implies& \textbf{F}(x+y) \nsucc \textbf{F}(x)\oplus\textbf{F}(y).
		\end{eqnarray*}
		Hence, $\textbf{F}$ is a sublinear IVF on $\mathcal{S}$.
	\end{example}
	\begin{example}\label{convexsub}
		Let $\mathcal{S}$ be a linear subspace of $\mathcal{X}$ and $\textbf{F}: \mathcal{S} \rightarrow I(\mathbb{R})$ be a convex
		IVF on $\mathcal{S}$  such that for all $x \in \mathcal{S}$,
		\begin{equation}\label{cnv}\textbf{F}(\alpha x) = \alpha \odot \textbf{F}(x)~\text{ for every}~ \alpha \geq 0.
		\end{equation}
Then, $\textbf{F}$ is a sublinear IVF on $\mathcal{S}$. The reason is as follows.\\
		For $x, y \in \mathcal{S}$ and $\lambda_1, \lambda_2 >0$, we have
		\begin{eqnarray*}
			\textbf{F}(\lambda_1x + \lambda_2 y)&=&
			\textbf{F}\Big(\lambda \Big( \tfrac{\lambda_1}{\lambda}x + \tfrac{\lambda_2}{\lambda}y \Big)\Big),~\text{where}~\lambda = \lambda_1+\lambda_2\\
			&=& \lambda \odot\textbf{F} \Big( \tfrac{\lambda_1}{\lambda}x + \tfrac{\lambda_2}{\lambda}y \Big)~\text{by}~(\ref{cnv})\\
			&\preceq& \lambda_1 \odot \textbf{F}(x) \oplus \lambda_2 \odot \textbf{F}(y)~\text{by the convexity of}~\textbf{F}.
		\end{eqnarray*}
		Taking $\lambda_1 = \lambda_2 = 1$, we obtain
		\[\textbf{F}(x+y) \preceq \textbf{F}(x) \oplus \textbf{F}(y)~\text{for all}~x, y \in \mathcal{S}.\]
		Hence, $\textbf{F}$ is a sublinear IVF on $\mathcal{S}$.
		\end{example}

		\begin{rmrk}\label{bq}
			A sublinear IVF may not be convex. For instance, take $\mathcal{X}$ as the Euclidean space $\mathbb{R}$, $\mathcal{S}=\mathcal{X}$ and the IVF $\textbf{F}: \mathcal{S} \rightarrow I(\mathbb{R})$ that is given by
		\[\textbf{F}(x) = \lvert x \rvert \odot [-3, 2].\]
		Clearly, by Example \ref{ex40}, $\textbf{F}$ is a sublinear IVF on $\mathcal{S}$. However, $\underline{f}(x)= -3\lvert x \rvert  $ is not convex on $\mathcal{S}$. Therefore, by Lemma \ref{lc1}, $\textbf{F}$ is not a convex IVF on $\mathcal{S}$.
		\end{rmrk}

	\begin{thm}\label{thm sub}
		Let $\mathcal{S}$ be a subset of $\mathcal{X}$ with nonempty interior, and let $\textbf{F}: \mathcal{S} \rightarrow I(\mathbb{R})$ be an IVF that is upper $gH$-Clarke differentiable at $\bar{x}\in int(\mathcal{S})$. Then, the upper $gH$-Clarke derivative $\textbf{F}_\mathscr{C}(\bar{x})$ of $\textbf{F}$ is a sublinear IVF on $\mathcal{S}$.
		
	\end{thm}
	
	\begin{proof}
		For  an arbitrary $h \in \mathcal{S}$ and $\alpha \geq 0$, we have
		\begin{eqnarray*}
		 \limsup_{\substack{%
					x \to \bar{x}\\
					\lambda \to 0+}}\frac{1}{\lambda}\odot\left(\textbf{F}(x+\lambda \alpha h)\ominus_{gH}\textbf{F}(x)\right)
			&=& \alpha\odot ( \limsup_{\substack{%
					x \to \bar{x}\\
					\lambda \to 0+}}\tfrac{1}{\lambda \alpha}\odot\left(\textbf{F}(x+\lambda \alpha h) \ominus_{gH} \textbf{F}(x)\right))\\
			&=& \alpha \odot \textbf{F}_\mathscr{C}(\bar{x})(h).
			\end{eqnarray*}
		Thus, $\textbf{F}_\mathscr{C}(\bar{x})(\alpha h)=\alpha \odot \textbf{F}_\mathscr{C}(\bar{x})(h).$\\
		Next, for all $h_1, h_2 \in \mathcal{S}$, we get
		\begingroup
		\allowdisplaybreaks
		\begin{eqnarray*}
			&&\textbf{F}_\mathscr{C}(\bar{x})( h_1+h_2)\\ &=& \limsup_{\substack{%
					x \to \bar{x}\\
					\lambda \to 0+}}\frac{1}{\lambda}\odot\left(\textbf{F}(x+\lambda( h_1+h_2))\ominus_{gH}\textbf{F}(x)\right)\\
				&\nsucc& \limsup_{\substack{%
					x \to \bar{x}\\
					\lambda \to 0+}}\frac{1}{\lambda}\odot\Big[\textbf{F}(x+\lambda h_1 + \lambda h_2)\ominus_{gH} \textbf{F}(x + \lambda h_2)\oplus \textbf{F}(x + \lambda h_2)\ominus_{gH} \textbf{F}(x) \Big],\\
				&&\text{by \eqref{aaaassaa} of Lemma \ref{forfrechet}}\\
			&=&\textbf{F}_\mathscr{C}(\bar{x})( h_1) \oplus \textbf{F}_\mathscr{C}(\bar{x})(h_2).	
			\end{eqnarray*}
			\endgroup
		Hence, $\textbf{F}_\mathscr{C}(\bar{x})$ is a sublinear IVF on $\mathcal{S}$.
	\end{proof}

	\section{Conclusion and Future Directions} \label{sect6}

In this article, mainly three concepts on IVFs have been studied---limit superior of IVF (Definition \ref{sup}), upper $gH$-Clarke derivative (Definition \ref{dffff}), and sublinear IVF (Definition \ref{sublinear_ivf}). One can trivially notice that in the degenerate case, each of the Definitions \ref{sup}, \ref{dffff}, and \ref{sublinear_ivf} reduces to the respective conventional definition for the real-valued functions. It has been observed that for a $gH$-Lipschitz continuous IVF, the upper $gH$-Clarke derivative always exists 
(Theorem \ref{thm_lipshitz_d}). Also, for a $gH$-Lipschitz continuous IVF, it has been found that the $gH$-directional derivative of a convex IVF coincides with the upper $gH$-Clarke derivative (Theorem \ref{_lipshitz}). 
It has been noticed that the upper $gH$-Clarke derivative at an interior point of the domain of an IVF is a sublinear IVF (Theorem \ref{thm sub}).\\

In analogy to the current study, future research can be carried out for other generalized directional derivatives for IVFs, e.g., Dini, Hadamard, Michel-Penot, etc., and their relationships \cite{demyanov2002rise}. In parallel to the proposed analysis of IVFs, another promising direction of future research can be the analysis of the fuzzy-valued functions (FVFs) as the alpha-cuts of fuzzy numbers are compact intervals \cite{ghosh2019introduction}. Hence, in future, one can attempt to extend the proposed idea of $gH$-Clarke derivative for  fuzzy-valued functions.  \\

The applications of the proposed upper $gH$-Clarke derivative in control systems and differential equations in a noisy or uncertain environment can be also dealt with in the future. 
A control system or a differential equation in noisy environment inevitably appears due to the incomplete information (e.g., demand for a product) or unpredictable changes (e.g., changes in the climate) in the system. The general control problem in a noisy or uncertain environment that we shall consider to study is the following:
\begin{align*}
 \min ~&~ x(u) \\
 \text{subject to} ~&~ \frac{dx}{dt} = \textbf{H}(x(t))\oplus \textbf{F}(u(t)), \\
                   ~&~ x(0) = [\alpha, \alpha],~x(t) \in I(\mathbb{R}),~u(t)\in U~\text{and}~t\in [0, T], 
\end{align*}
where $x:[0, T] \rightarrow I(\mathbb{R})$ is the unknown function, $u : [0, T] \rightarrow U \subset \mathbb{R}$ is the control variable, $T\in\mathbb{R}_{+}$ and $\alpha \in \mathbb{R} $. Here, $\textbf{H}$ and $\textbf{F}$ are upper $gH$-Clarke and $gH$-Fr\'echet differentiable IVFs with respect to the control variable $u$. In such a control problem, we shall show the usefulness of the proposed upper $gH$-Clarke derivative to find the optimal control of the system.\\

	\appendix
	 \section{Proof of Lemma \ref{last}} \label{appendix_co}
	\begin{proof} Let $\textbf{C}=[\underline{c},~\overline{c}]$.
	\begin{enumerate}[(i)]
	\item If $\textbf{C}\succeq \textbf{0}$, then
		\begingroup
		\allowdisplaybreaks
	\begin{align*}
	&\underline{c}\geq 0~\text{and}~\overline{c}\geq 0\\
	\implies&\lvert x \rvert \underline{c}+\lvert y \rvert \underline{c} \geq \lvert x+y \rvert \underline{c}~\text{and}~\lvert x \rvert \overline{c}+\lvert y \rvert \overline{c} \geq \lvert x+y \rvert \overline{c}\\
	\implies& \lvert x+y \rvert \odot \textbf{C} \preceq \lvert x \rvert \odot \textbf{C} \oplus  \lvert y \rvert \odot \textbf{C}.
	\end{align*}
	
	\item If $\textbf{C}\preceq \textbf{0}$, then
	\begin{align*}
	&\underline{c}\leq 0~\text{and}~\overline{c}\leq 0\\
	\implies&\lvert x \rvert \underline{c}+\lvert y \rvert \underline{c} \leq \lvert x+y \rvert \underline{c}~\text{and}~\lvert x \rvert \overline{c}+\lvert y \rvert \overline{c} \leq \lvert x+y \rvert \overline{c}\\
	\implies& \lvert x+y \rvert \odot \textbf{C} \succeq \lvert x \rvert \odot \textbf{C} \oplus  \lvert y \rvert \odot \textbf{C}.
	\end{align*}
	
	\end{enumerate}
	\endgroup
	\end{proof}
	
	\section{Proof of Lemma \ref{forfrechet}} \label{appendix frechet}
\begin{proof} Let $\textbf{A} = [\underline{a}, \overline{a}], ~\textbf{B} = [\underline{b}, \overline{b}]$ and $\textbf{C} = [\underline{c}, \overline{c}]$.\\
\begin{enumerate}[(i)]
\item\label{cgh} We have the following four possible cases.
		\begin{enumerate}[$\bullet$ \textbf{Case} 1.]
			\item Let $\overline{a}-\overline{c}\geq\underline{a}-\underline{c}$ and $\overline{c}-\overline{b}\geq\underline{c}-\underline{b}$.
			Then, $\overline{a}-\overline{b}\geq\underline{a}-\underline{b}$ and
			\[
			 (\textbf{A}\ominus_{gH}\textbf{C})\oplus(\textbf{C}\ominus_{gH}\textbf{B}) = [\underline{a}-\underline{c}, \overline{a}-\overline{c}]\oplus[\underline{c}-\underline{b},\overline{c}-\overline{b}] =[\underline{a}-\underline{b},\overline{a}-\overline{b}]=\textbf{A}\ominus_{gH}\textbf{B}.
			\]

			\item Let $\overline{a}-\overline{c}\leq\underline{a}-\underline{c}$ and $\overline{c}-\overline{b}\leq\underline{c}-\underline{b}$.
			Therefore, $\overline{a}-\overline{b}\leq\underline{a}-\underline{b}$ and
			\[
			 (\textbf{A}\ominus_{gH}\textbf{C})\oplus(\textbf{C}\ominus_{gH}\textbf{B}) = [\overline{a}-\overline{c}, \underline{a}-\underline{c}]\oplus[\overline{c}-\overline{b},\underline{c}-\underline{b}] =[\overline{a}-\overline{b},\underline{a}-\underline{b}]=\textbf{A}\ominus_{gH}\textbf{B}.
			\]

			\item Let $\overline{a}-\overline{c}<\underline{a}-\underline{c}$ and $\overline{c}-\overline{b}>\underline{c}-\underline{b}$.
			Therefore,
			\[
			 (\textbf{A}\ominus_{gH}\textbf{C})\oplus(\textbf{C}\ominus_{gH}\textbf{B}) = [\overline{a}-\overline{c}, \underline{a}-\underline{c}]\oplus[\underline{c}-\underline{b},\overline{c}-\overline{b}] =[\overline{a}-\overline{c}+\underline{c}-\underline{b}, \underline{a}-\underline{c}+\overline{c}-\overline{b}].
			\]
			If possible, let
			\begin{equation}\label{zxcvc}
				 (\textbf{A}\ominus_{gH}\textbf{C})\oplus(\textbf{C}\ominus_{gH}\textbf{B}) \prec \textbf{A}\ominus_{gH}\textbf{B}.
			\end{equation}
			If $\overline{a}-\overline{b}\geq\underline{a}-\underline{b}$, then from \eqref{zxcvc} we get
			\begin{eqnarray*}
				 &&[\overline{a}-\overline{c}+\underline{c}-\underline{b}, \underline{a}-\underline{c}+\overline{c}-\overline{b}] \prec[\underline{a}-\underline{b},\overline{a}-\overline{b}]\\
				&\Longrightarrow& \underline{a}-\underline{c}+\overline{c}-\overline{b} \leq \overline{a}-\overline{b}\\
				&\Longrightarrow& \underline{a}-\underline{c} \leq \overline{a}-\overline{c},~\text{ which is an impossibility}.
			\end{eqnarray*}
			Further, if $\overline{a}-\overline{b}\leq\underline{a}-\underline{b}$, then from \eqref{zxcvc}, we have
				\begingroup
		\allowdisplaybreaks
			\begin{eqnarray*}
				 &&[\overline{a}-\overline{c}+\underline{c}-\underline{b}, \underline{a}-\underline{c}+\overline{c}-\overline{b}] \prec[\overline{a}-\overline{b},\underline{a}-\underline{b}]\\
				&\Longrightarrow& \underline{a}-\underline{c}+\overline{c}-\overline{b} \le \underline{a}-\underline{b}\\
				&\Longrightarrow& \overline{c}-\overline{b} \le \underline{c}-\underline{b},~\text{which is an impossibility}.
			\end{eqnarray*}
		 Thus, \eqref{zxcvc} is not true.
			
			\item Let $\overline{a}-\overline{c}>\underline{a}-\underline{c}$ and $\overline{c}-\overline{b}<\underline{c}-\underline{b}$. Proceeding as in \textbf{Case} $3$ of \eqref{cgh} we can prove that \eqref{zxcvc} is not true. Hence,
		\end{enumerate}
		\[
		 (\textbf{A}\ominus_{gH}\textbf{C})\oplus(\textbf{C}\ominus_{gH}\textbf{B}) \nprec \textbf{A}\ominus_{gH}\textbf{B}.
		\]
		\endgroup
 
	\item As
		${\lVert \textbf{B} \ominus_{gH} \textbf{A} \rVert}_{I(\mathbb{R})} = \max \{|\underline{b}-\underline{a}|, |\overline{b}-\overline{a}|\},$ we break the proof in two cases.

		\begin{enumerate}[$\bullet$ \textbf{Case} 1.]
			\item  If $(L = )~ {\lVert \textbf{B} \ominus_{gH} \textbf{A} \rVert}_{I(\mathbb{R})} = |\underline{b}-\underline{a}|$, then
			\begin{equation}\label{align}
			  |\underline{b}-\underline{a}| \geq |\overline{b}-\overline{a}| \implies |\underline{b}-\underline{a}| \geq \overline{b}-\overline{a} \implies   \overline{b} \leq \overline{a}+L.
			\end{equation}
			
			Since $ \underline{b}-\underline{a} \leq |\underline{b}-\underline{a}|$, then
			\begin{equation}\label{jkl}
				\underline{b} \leq \underline{a}+L.
			\end{equation}
			From (\ref{align}) and (\ref{jkl}), we have
			\[\textbf{B} \preceq \textbf{A}\oplus [L, L].\]
			
			\item If $(L = )~ {\lVert \textbf{B} \ominus_{gH} \textbf{A} \rVert}_{I(\mathbb{R})} = |\overline{b}-\overline{a}|$, then
			\begin{equation}\label{ymmm}
				 |\underline{b}-\underline{a}| \leq |\overline{b}-\overline{a}| \implies \underline{b}-\underline{a} \leq |\overline{b}-\overline{a}|  \implies \underline{b} \leq \underline{a}+L.
			\end{equation}
			Since $ \overline{b}-\overline{a} \leq |\overline{b}-\overline{a}|$,
			\begin{equation}\label{nnnn}
				\overline{b} \leq \overline{a}+L.
			\end{equation}
			From (\ref{ymmm}) and (\ref{nnnn}), we obtain
			\[\textbf{B} \preceq \textbf{A}\oplus [L, L], ~\text{where}~ L=\lVert \textbf{B}\ominus_{gH}\textbf{A} \rVert_{I(\mathbb{R})}.\]
		\end{enumerate}

		\item\label{cvbvbb} If possible, let there exist $\textbf{A},~\textbf{B},~\textbf{C}$ and $\textbf{D}$ in $I(\mathbb{R})$ such that
		\begin{equation}\label{impor}
			{\lVert(\textbf{A}\ominus_{gH}\textbf{B})\ominus_{gH} (\textbf{C}\ominus_{gH}\textbf{D}) \rVert}_{I(\mathbb{R})}
			>
			\lVert \textbf{A}\ominus_{gH}\textbf{C} \rVert_{I(\mathbb{R})} \oplus \lVert \textbf{B}\ominus_{gH}\textbf{D} \rVert_{I(\mathbb{R})}.
		\end{equation}
		According to the definition of $gH$-difference of two intervals,	
		\begin{equation}\label{p3}
			\text{either}~ \textbf{A}\ominus_{gH} \textbf{B}= [\underline{a}-\underline{b}, \overline{a}-\overline{b}] ~\text{or}~ \textbf{A}\ominus_{gH} \textbf{B}= [\overline{a}-\overline{b}, \underline{a}-\underline{b}],
		\end{equation}
		\begin{equation*}
			\text{either}~ \textbf{C}\ominus_{gH} \textbf{D}= [\underline{c}-\underline{d}, \overline{c}-\overline{d}]~\text{or}~\textbf{C}\ominus_{gH} \textbf{D}= [\overline{c}-\overline{d}, \underline{c}-\underline{d}],
		\end{equation*}
		\begin{equation}\label{p11}
			\text{either}~ \textbf{A}\ominus_{gH} \textbf{C}= [\underline{a}-\underline{c}, \overline{a}-\overline{c}] ~\text{or}~ \textbf{A}\ominus_{gH} \textbf{B}= [\overline{a}-\overline{c}, \underline{a}-\underline{c}],
		\end{equation}
		and
		\begin{equation*}
			\text{either}~ \textbf{B}\ominus_{gH} \textbf{D}= [\underline{b}-\underline{d}, \overline{b}-\overline{d}]~\text{or}~\textbf{B}\ominus_{gH} \textbf{D}= [\overline{b}-\overline{d}, \underline{b}-\underline{d}].
		\end{equation*}
	Then, one of the following holds true:
		\begin{enumerate}[(a)]
			\item $(\textbf{A} \ominus_{gH} \textbf{B})\ominus_{gH}(\textbf{C}\ominus_{gH} \textbf{D})= [\underline{a}-\underline{b}-\underline{c}+\underline{d},~ \overline{a}-\overline{b}-\overline{c}+\overline{d} ]$

			\item $(\textbf{A} \ominus_{gH} \textbf{B})\ominus_{gH}(\textbf{C}\ominus_{gH} \textbf{D})= [\underline{a}-\underline{b}-\overline{c}+\overline{d},~ \overline{a}-\overline{b}-\underline{c}+\underline{d} ]$
			
			\item $(\textbf{A} \ominus_{gH} \textbf{B})\ominus_{gH}(\textbf{C}\ominus_{gH} \textbf{D})= [ \overline{a}-\overline{b}-\overline{c}+\overline{d},~\underline{a}-\underline{b}-\underline{c}+\underline{d} ]$

			\item $(\textbf{A} \ominus_{gH} \textbf{B})\ominus_{gH}(\textbf{C}\ominus_{gH} \textbf{D})= [\overline{a}-\overline{b}-\underline{c}+\underline{d},~ \underline{a}-\underline{b}-\overline{c}+\overline{d} ]$
			
		\end{enumerate}
\begin{enumerate}[$\bullet$ \textbf{Case} 1.]
\item \label{case-1}
Let $(\textbf{A} \ominus_{gH} \textbf{B})\ominus_{gH}(\textbf{C}\ominus_{gH} \textbf{D})= [\underline{a}-\underline{b}-\underline{c}+\underline{d},~ \overline{a}-\overline{b}-\overline{c}+\overline{d} ]$.
\begin{enumerate}[(a)]
 \item If ${\lVert(\textbf{A}\ominus_{gH}\textbf{B})\ominus_{gH} (\textbf{C}\ominus_{gH}\textbf{D}) \rVert}_{I(\mathbb{R})} = \lvert \underline{a}-\underline{b}-\underline{c}+\underline{d} \rvert$, then from equation (\ref{impor}), we have
 \begin{equation*}
 \lvert \underline{a}-\underline{b}-\underline{c}+\underline{d} \rvert > \lvert \underline{a}-\underline{c}\rvert + \lvert \underline{b}-\underline{d}\rvert > \lvert \underline{a}-\underline{b}-\underline{c}+\underline{d} \rvert,
 \end{equation*}
 which is impossible.

 \item If ${\lVert(\textbf{A}\ominus_{gH}\textbf{B})\ominus_{gH} (\textbf{C}\ominus_{gH}\textbf{D}) \rVert}_{I(\mathbb{R})} = \lvert \overline{a}-\overline{b}-\overline{c}+\overline{d} \rvert$, then from equation (\ref{impor}), we have
 \begin{equation*}
 \lvert \overline{a}-\overline{b}-\overline{c}+\overline{d} \rvert > \lvert \overline{a}-\overline{c}\rvert + \lvert \overline{b}-\overline{d}\rvert > \lvert \overline{a}-\overline{b}-\overline{c}+\overline{d} \rvert,
 \end{equation*}
 which is again impossible.
\end{enumerate}

\item \label{case_2}
Let $(\textbf{A} \ominus_{gH} \textbf{B})\ominus_{gH}(\textbf{C}\ominus_{gH} \textbf{D})= [\overline{a}-\overline{b}-\overline{c}+\overline{d},~ \underline{a}-\underline{b}-\underline{c}+\underline{d} ]$.\\
For this case, two subcases are similar to the \textbf{Case} $1$ of \eqref{cvbvbb} will lead to impossibilities.

\item \label{case_3}
Let $(\textbf{A} \ominus_{gH} \textbf{B})\ominus_{gH}(\textbf{C}\ominus_{gH} \textbf{D})= [\underline{a}-\underline{b}-\overline{c}+\overline{d},~ \overline{a}-\overline{b}-\underline{c}+\underline{d} ]$. Then,
\begin{equation}\label{eqn12}
\underline{a}-\underline{b} \leq \overline{a}-\overline{b}~\text{and}~\overline{c}-\overline{d} \leq \underline{c}-\underline{d}.
\end{equation}

\begin{enumerate}[(a)]
 \item If ${\lVert(\textbf{A}\ominus_{gH}\textbf{B})\ominus_{gH} (\textbf{C}\ominus_{gH}\textbf{D}) \rVert}_{I(\mathbb{R})} = \lvert \overline{a}-\overline{b}-\underline{c}+\underline{d} \rvert$, then $\overline{a}-\overline{b}-\underline{c}+\underline{d} \geq 0.$
 From equation (\ref{impor}), we have
 \begin{equation*}
 \lvert \overline{a}-\overline{b}-\underline{c}+\underline{d} \rvert > \lvert \overline{a}-\overline{c}\rvert + \lvert \overline{b}-\overline{d}\rvert \implies \overline{c}-\overline{d} > \underline{c}-\underline{d},
 \end{equation*}
 which is contradictory to (\ref{eqn12}).

 \item If ${\lVert(\textbf{A}\ominus_{gH}\textbf{B})\ominus_{gH} (\textbf{C}\ominus_{gH}\textbf{D}) \rVert}_{I(\mathbb{R})} = \lvert \underline{a}-\underline{b}-\overline{c}+\overline{d} \rvert$, then $\underline{a}-\underline{b}-\overline{c}+\overline{d} < 0.$
 From equation (\ref{impor}), we have
 \begin{equation*}
-(\underline{a}-\underline{b}-\overline{c}+\overline{d}) = \lvert \underline{a}-\underline{b}-\overline{c}+\overline{d} \rvert > \lvert \underline{a}-\underline{c}\rvert + \lvert \underline{b}-\underline{d}\rvert \implies \overline{c}-\overline{d} > \underline{c}-\underline{d},
 \end{equation*}
 which is again contradictory to (\ref{eqn12}).
\end{enumerate}

\item \label{case_4}
Let $(\textbf{A} \ominus_{gH} \textbf{B})\ominus_{gH}(\textbf{C}\ominus_{gH} \textbf{D})= [\overline{a}-\overline{b}-\underline{c}+\underline{d},~\underline{a}-\underline{b}-\overline{c}+\overline{d} ]$.\\
All the two subcases for this case are similar to  \textbf{Case} $3$ of \eqref{cvbvbb}.
\end{enumerate}
Hence, (\ref{impor}) is wrong, and thus the result follows.
\end{enumerate}
	\end{proof}
\section{Proof of Lemma \ref{lc2}} \label{aplc1}
\begin{proof}
\begin{enumerate}[(i)]
    \item Let $\textbf{F}$ be $gH$-continuous at $\bar{x} \in \mathcal{S}$. Thus, for any $d\in\mathbb{R}^n$ such that $\bar{x}+d\in\mathcal{S}$,
\[
\lim_{\lVert d \rVert\to 0}\left(\textbf{F}(\bar{x}+d)\ominus_{gH}\textbf{F}(\bar{x})\right)=\textbf{0},
\]
which implies
\[
\lim_{\lVert d \rVert\to 0}(\underline{f}(\bar{x}+d)-\underline{f}(\bar{x}))\to 0\ \mbox{and}\ \lim_{\lVert d \rVert\to 0}(\overline{f}(\bar{x}+d)-\overline{f}(\bar{x}))\to 0,
\]
i.e., $\underline{f}$ and $\overline{f}$ are continuous at $\bar{x}\in\mathcal{S}$.\\

Conversely, let the functions $\underline{f}$ and $\overline{f}$ be continuous at $\bar{x}\in\mathcal{S}$. If possible, let $\textbf{F}$ be not $gH$-continuous at $\bar{x}$. Then, as $\lVert d \rVert\to 0,\ (\textbf{F}(\bar{x}+d)\ominus_{gH}\textbf{F}(\bar{x}))\not\to\textbf{0}$. Therefore, as $\lVert d \rVert\to 0$ at least one of the functions $(\underline{f}(\bar{x}+d)-\underline{f}(\bar{x}))$ and $(\overline{f}(\bar{x}+d)-\overline{f}(\bar{x}))$ does not tend to $0$. So it is clear that at least one of the functions $\underline{f}$ and $\overline{f}$ is not continuous at $\bar{x}$. This contradicts the assumption that the functions $\underline{f}$ and $\overline{f}$ both are continuous at $\bar{x}$. Hence, $\textbf{F}$ is $gH$-continuous at $\bar{x}$. \\
\item Let $\textbf{F}$ be $gH$-Lipschitz continuous on $\mathcal{S}$. Thus, there exists $K>0$ such that for any $x,~y \in \mathcal{X}$ we have
		\begin{eqnarray*}
			&&{\lVert \textbf{F}(x) \ominus_{gH} \textbf{F}(y) \rVert}_{I(\mathbb{R})} \leq K {\lVert x-y \rVert}\\
		&\implies& \left \lvert \underline{f}(x) - \underline{f}(y) \right\rvert \leq K \lVert  x-y \rVert~\text{and}~\left\lvert \overline{f}(x) - \overline{f}(y) \right\rvert \leq K \lVert  x-y \rVert.
		\end{eqnarray*}
Hence, $\underline{f}$ and $\overline{f}$ are Lipschitz continuous on $\mathcal{S}$.\\
Conversely, let the functions $\underline{f}$ and $\overline{f}$ be Lipschitz continuous on $\mathcal{S}$. Thus, there exist $K_1,~K_2>0$ such that for all $x,~y \in \mathcal{S}$,
\begin{eqnarray*}
				&& \left \lvert \underline{f}(x) - \underline{f}(y) \right\rvert \leq K_1 \lVert  x-y \rVert~\text{and}~\left\lvert \overline{f}(x) - \overline{f}(y) \right\rvert \leq K_2 \lVert  x-y \rVert\\
		&\implies& \max \left\{\left \lvert \underline{f}(x) - \underline{f}(y) \right\rvert,~\left\lvert \overline{f}(x) - \overline{f}(y) \right\rvert\right\} \leq \bar K \lVert x-y \rVert,~\text{where}~\bar K=\max\{K_1,~K_2\}\\
		&\implies& {\lVert \textbf{F}(x) \ominus_{gH} \textbf{F}(y) \rVert }_{I(\mathbb{R})} \le \bar K {\lVert x-y \rVert}.
		\end{eqnarray*}
		Hence, $\textbf{F}$ is $gH$-Lipschitz continuous IVF on $\mathcal{S}$.
\item        Let $\textbf{F}$ be $gH$-Lipschitz continuous on $\mathcal{S}$. Then, there exists an $K>0$ such that for all $x,~y \in \mathcal{S}$, we have
        \[ {\lVert \textbf{F}(y) \ominus_{gH} \textbf{F}(x) \rVert }_{I(\mathbb{R})} \le K {\lVert y-x \rVert}.\]
        For $h=y-x\in \mathcal{S}$,
        \begin{eqnarray*}
        &&{\lVert \textbf{F}(x+h) \ominus_{gH} \textbf{F}(x) \rVert }_{I(\mathbb{R})} \le K {\lVert h \rVert}\\
        &\implies& \lim_{\lVert h \rVert\rightarrow 0}{\lVert \textbf{F}(x+h) \ominus_{gH} \textbf{F}(x) \rVert }_{I(\mathbb{R})}=0\\
        &\implies&\lim_{\lVert h \rVert\rightarrow 0}\left( \textbf{F}(x+h) \ominus_{gH} \textbf{F}(x)\right)=\textbf{0}.
        \end{eqnarray*}
        Hence, $\textbf{F}$ is $gH$-continuous at $x\in \mathcal{S}$.
       \end{enumerate} 
       \end{proof}

	\section{Proof of Lemma \ref{00}} \label{appendix_sup}
	\begin{proof}
	\begin{enumerate}[(i)]
		\item 
		Since \[\limsup\limits_{x\to \bar{x}}\left(\underline{f}(x)+\underline{g}(x)\right) \leq \limsup\limits_{x\to \bar{x}} \underline{f}(x) + \limsup\limits_{x\to \bar{x}} \underline{g}(x)~\text{and}\] \[\limsup\limits_{x\to \bar{x}}\left(\overline{f}(x)+\overline{g}(x)\right) \leq \limsup\limits_{x\to \bar{x}}\overline{f}(x)+ \limsup\limits_{x\to \bar{x}} \overline{g}(x),\]
		then
			\begingroup
		\allowdisplaybreaks
		\begin{eqnarray*}
			&&\left[\limsup\limits_{x\to \bar{x}}(\underline{f}(x)+\underline{g}(x)), ~ \limsup\limits_{x\to \bar{x}}\left(\overline{f}(x)+\overline{g}(x)\right)\right]\\
			&\preceq& \left[\limsup\limits_{x\to \bar{x}} \underline{f}(x), \limsup\limits_{x\to \bar{x}} \overline{f}(x) \right]\oplus \left[\limsup\limits_{x\to \bar{x}} \underline{g}(x), \limsup\limits_{x\to \bar{x}} \overline{g}(x) \right],  
			\end{eqnarray*}
which implies 
			$\limsup\limits_{x\to \bar{x}}\left(\textbf{F}(x)\oplus \textbf{G}(x)\right)\preceq \limsup\limits_{x\to \bar{x}} \textbf{F}(x) \oplus \limsup\limits_{x\to \bar{x}} \textbf{G}(x).$

		\item  Since $\underline{f}$ and $\overline{f}$ are real-valued functions, for any $\lambda \geq 0$, we have 
		\begin{equation}\label{ffff}
		\limsup\limits_{x\to \bar{x}}\left(\lambda \underline{f}(x)\right) = \lambda \limsup\limits_{x\to \bar{x}} \underline{f}(x)~\text{and}~ \limsup\limits_{x\to \bar{x}}\left(\lambda \overline{f}(x)\right) = \lambda \limsup\limits_{x\to \bar{x}} \overline{f}(x).
		\end{equation}
		Hence, for any $\lambda \geq 0$, 
		\begin{eqnarray*}
		\limsup\limits_{x\to \bar{x}}\left(\lambda \odot \textbf{F}(x)\right) &=& \left[\limsup\limits_{x\to \bar{x}}\left(\lambda \underline{f}(x)\right), ~ \limsup\limits_{x\to \bar{x}}\left( \lambda \overline{f}(x)\right)\right]\\
		&=&	\lambda \odot \limsup\limits_{x\to \bar{x}}  \textbf{F}(x)~\text{by  (\ref{ffff})}.
		\end{eqnarray*}
		
        \item  Let $f$ be a real-valued function. Then, $\left\lvert \limsup\limits_{x\to \bar{x}} f(x) \right\rvert \leq \limsup\limits_{x\to \bar{x}} \left\lvert f(x) \right\rvert$. By the definition of norm on $I(\mathbb{R})$,
		\begin{eqnarray*}
			\left\lVert \limsup\limits_{x\to \bar{x}} \textbf{F}(x) \right\rVert_{I(\mathbb{R})} &=& \max \left\{\left\lvert \limsup\limits_{x\to \bar{x}} \underline{f}(x) \right\rvert ,~\left\lvert \limsup\limits_{x\to \bar{x}} \overline{f}(x) \right\rvert\right\} \\
			&\leq&\limsup\limits_{x\to \bar{x}}~ \lVert  \textbf{F}(x) \rVert_{I(\mathbb{R})}.
		\end{eqnarray*}
		\end{enumerate}
		\endgroup
		\end{proof}

\section{Proof of Lemma \ref{clarkeimply}} \label{appendix_imply}
	\begin{proof}
	Since $\underline{f}$ and $\overline{f}$ are upper Clarke differentiable at $\bar{x}$. Therefore, both of the following limits 
	\[\limsup_{\substack{%
					x \to \bar{x}\\
					\lambda \to 0+}}\frac{1}{\lambda}l_1(\lambda)~\text{and}~\limsup_{\substack{%
					x \to \bar{x}\\
					\lambda \to 0+}}\frac{1}{\lambda}l_2(\lambda)
	\]
	exist, where $l_1(\lambda)=\underline{f}(x+\lambda h)-\underline{f}(x)$ and $l_2(\lambda)=\overline{f}(x+\lambda h)-\overline{f}(x)$. Thus,
	\begingroup
	    \allowdisplaybreaks
		\begin{eqnarray*}
			&&\limsup_{\substack{%
					x \to \bar{x}\\
					\lambda \to 0+}}\frac{1}{\lambda}\left(l_1(\lambda)+l_2(\lambda)\right)~\text{and}~\limsup_{\substack{%
					x \to \bar{x}\\
					\lambda \to 0+}}\frac{1}{\lambda}\lvert l_1(\lambda)-l_2(\lambda)\rvert~\text{exist}\\
			&\implies&\limsup_{\substack{%
					x \to \bar{x}\\
					\lambda \to 0+}}\frac{1}{2\lambda}\Big(l_1(\lambda)+l_2(\lambda)-\lvert l_1(\lambda)-l_2(\lambda)\rvert\Big)~\text{and}\\
			&&\limsup_{\substack{%
					x \to \bar{x}\\
					\lambda \to 0+}}\frac{1}{2\lambda}\Big(l_1(\lambda)+l_2(\lambda)+\lvert l_1(\lambda)-l_2(\lambda)\rvert\Big)~\text{exist}\\
			&\implies& \limsup_{\substack{%
					x \to \bar{x}\\
					\lambda \to 0+}}\frac{1}{\lambda} \left(\min\left\{l_1(\lambda), l_2(\lambda) \right\}\right)~\text{and}~\limsup_{\substack{%
					x \to \bar{x}\\
					\lambda \to 0+}}\frac{1}{\lambda}\left(\max\left\{l_1(\lambda), l_2(\lambda) \right\}\right)~\text{exist}\\
			&\implies&\limsup_{\substack{%
					x \to \bar{x}\\
					\lambda \to 0+}}\frac{1}{\lambda}\odot\left(\textbf{F}(x+\lambda h)\ominus_{gH}\textbf{F}(x)\right)~\text{exists}.
			\end{eqnarray*}
			\endgroup
			Hence, $\textbf{F}$ is upper $gH$-Clarke differentiable IVF at $\bar{x} \in \mathcal{S}$. 
    	\end{proof}


$\\$
\noindent
\textbf{Acknowledgement}\\ \\
The first author is thankful for a research scholarship awarded by the University Grants Commission, Government of India. \\

\end{document}